\documentclass[%
 reprint,
nofootinbib,
 amsmath,amssymb,
 aps,
]{revtex4-2}

\usepackage{graphicx}
\usepackage{dcolumn}
\usepackage{bm}
\usepackage{dsfont}
\usepackage{amsthm}
\usepackage{color}
\usepackage[normalem]{ulem}
\usepackage[colorlinks=true, allcolors=blue]{hyperref}

\newcommand{\R}{\mathds{R}}

\newcommand{\M}{\mathcal{M}}
\newcommand{\G}{\mathcal{G}}
\newcommand{\at}[2][]{#1|_{#2}}

\usepackage{indentfirst}

\theoremstyle{thmstyletwo}%
\newtheorem{theorem}{Theorem}
\newtheorem{definition}{Definition}

\begin{document}


\title{From local isometries to global symmetries: \\ Bridging Killing vectors and Lie algebras through induced vector fields}

\author{Thales B. S. F. Rodrigues}
 \email{thalesfonseca@ice.ufjf.br}
 \affiliation{%
 Departamento de Física, Programa de Pós-Graduação em Física, Universidade Federal de Juiz de Fora
}%
\author{B. F. Rizzuti}%
 \email{brunorizzuti@ice.ufjf.br}
\affiliation{%
 Departamento de Física, Programa de Pós-Graduação em Física, Universidade Federal de Juiz de Fora
}%


\begin{abstract}
The study of symmetries in the realm of manifolds can be approached in two different ways. On one hand, Killing vector fields on a (pseudo-)Riemannian manifold correspond to the directions of local isometries within it. On the other hand, from an algebraic perspective, global symmetries of such manifolds are associated with group elements. Although the connection between these two concepts is well established in the literature, this work aims to build an unexplored bridge between the Killing vector fields of $n$-dimensional maximally symmetric spaces and their corresponding isometry Lie groups, anchoring primarily on the definition of induced vector fields. As an application of our main result, we explore two specific examples: the three-dimensional Euclidean space and Minkowski spacetime.
\end{abstract}

\maketitle


\section{Introduction}\label{Sec1}

On pure geometric grounds, Killing vector fields \cite{nakahara_geometry_1990} play a central role in Riemanian manifolds. They serve as elements of tangent planes, generating local isometries and, in some approaches, could be used to construct global conserved structures on general relativistic spacetimes \cite{Feng_2018}. Its arms extends to different areas, ranging from classical mechanics \cite{rizzuti_square_2019} to the description of constant-curvature spaces of a homogeneous isotropic and static spacetimes \cite{Katanaev, dobarro_characterizing_2012} passing through, recently, to their meaning in comprehending fluid flows on curved surfaces \cite{shimizu_hydrodynamic_2024}. According to \cite{SANCHEZ1997643}, there are several approaches to study these vector fields on Lorentzian manifolds (i.e., pseudo-Riemannian manifolds which admit a pseudo-metric $\eta$). One of them is to employ proper actions of Lie groups. However, it's not explicitly addressed how to recover the Lie group using the expression of Killing vector fields, which is the case in \cite{kobayashi_transformation_1972} as  well.   

On the other side, Lie groups \cite{hall_lie_2015} act on manifolds, inducing a linear representation on its tangent space, representing metric-preserving global transformations. Far from being its only application, Lie groups and algebras reach distinct branches of both geometric and algebraic settings, see, for example, \cite{benayadi_pseudo-euclidean_2023} and references therein.

There are interesting results relating these two approaches ---by  one side a local geometrical isometry and, on the other, the global metric preservation. For example, it is known that the set of Killing vector fields on an manifold $\M$ is indeed a Lie algebra \cite{alexandrino_lie_2015}. Even more, for the case of curved spaces, if the Riemann curvature tensor vanishes at some point $p\in \M$, then the Lie algebra of Killing vector fields is related to a subalgebra of $SO(r,q)$, where the pair $(r,q)$ denotes the signature of the local  pseudo-metric of the tangent space at $p$ \cite{atkins_lie_2008}. Hence, leveraging on the previous cited results, our main point of investigation lies on the study of, up to our knowledge, a not well explored connection between these two topics. 

More specifically, our investigation takes a significant turn as we identify a connection between Killing vector fields of $(\R^n,\delta)$ and $(\R^{1 + n},\eta)$ (for $n\geq 2$)\footnote{In our construction, we arrive at certain specific symmetry groups, which requires imposing restrictions on the dimension $n$. Without these restrictions, it would be impossible to reconstruct the rotation generators in a general manner across both manifolds of interest. For example, discussing spatial rotations in $(\R, \delta)$ or $(\R^2, \eta)$ would be senseless.}  and Lie algebra elements by interpreting them as induced vector fields, with the definition of the latter being presented in \cite{nakahara_geometry_1990}. Furthermore, the connection we demonstrate here between this definition and Killing vector fields as generators of Lie algebras, to the best of our knowledge, has not been thoroughly investigated in the literature. We delve into specific examples of our construction (section \ref{Sec3}), namely, $(\R^3, \delta)$ and $(\R^{4}, \eta)$. This, in turn, when identifying the Killing vector fields of these manifolds as induced vector fields, allows us to reconstruct the Lie groups that preserve the respective metrics, which are, in the first instance, the special Euclidean group $SE(3)$, whereas the Poincaré group emerges in the second case. 

Our findings contribute to a deeper understanding of the interplay between geometry and group theory, reinforcing the bridge between the intrinsic geometry of manifolds and the algebraic structures governing their symmetries.  Thus, we summarize our main result in the following theorem, that will be proved in the next section. 

\begin{theorem}
\label{theo:1}
   Let $(\R^n, \delta)$ and $(\R^{1+n}, \eta)$ represent the $n-$dimensional Euclidean space and the $1+n-$dimensional Minkowski spacetime ($n\geq 2$), respectively. Suppose we define their Killing vector fields as induced vector fields by Lie algebra elements. Then, they generate the corresponding isometry groups: the special Euclidean group $SE(n)$ in the former and the Poincaré group in the latter.
\end{theorem}

Our notation will be the following. When referring to arbitrary indices of objects on \(n\)-dimensional manifolds, Latin letters $i,j,k,\dots$ run from $1$ to $n$. Greek letters $\alpha, \beta, \mu, \dots$ run from $0$ to $n$. In some contexts, which will be specified throughout the text, the letters $i,j,k, \dots$ will run from $1$ to $3$, whilst $\alpha, \beta, \mu, \dots$ will run from $0$ to $3$. Other exceptions will be pinpointed whenever necessary. Einstein notation of summation is employed throughout the text, and we restrict our discussion to matrix groups only. The $\eta$ signature adopted was $(-1,1,1,\dots,1)$. 

\section{Proof of Theorem \ref{theo:1}}\label{Sec2}

In this section, we will demonstrate the main result of this work, which is based on the connection between Killing vectors and Lie algebras through the definition of induced vector fields. The concept of induced vector fields and their relation to Lie algebra elements, which is central to your construction, is a standard topic in the literature that focuses on the intersection of differential geometry and Lie group theory \cite{LeeSmooth, nakahara_geometry_1990, differential_geometry_Lie_groups}. Here, for the proof of our Theorem \ref{theo:1}, we present the following definition of induced vector fields, which is based on an extended discussion found in \cite{nakahara_geometry_1990}.

\begin{definition}\label{def.induced.vector}
Let $\Sigma: \G \times \M \rightarrow \M$ be an action of the Lie group  $\G$ on a manifold $\M$. Given $V \in T_e\G$ and a point $p \in \M$, define a flow in $\M$ by  $$ \sigma(t,p):= \Sigma \left( e^{tV},p\right ).$$
The left-invariant vector $X_V$ generated by $V$ induces the following vector field, which is, basically, a tangent vector to the curve $\sigma(t,p)$
$$ V^I\vert_p := \frac{d}{dt} \sigma(t,p)\Big \vert_{t=0} = \frac{d}{dt}\Sigma \left( e^{tV},p\right )\Big \vert_{t=0}. $$
The map $I: T_e \G \rightarrow \mathfrak{X}(\M)$ defines what we call the induced vector field $V^I$. 
\end{definition}
In what follows, we proceed to the proof. 


\begin{proof}[Proof of Theorem~\ref{theo:1}]
Suppose, firstly, that the Killing vector fields $\mathcal{J}_{ij} = x^i \partial_j - x^j \partial_i$ ($i$ and $j$ cyclic) of $(\R^n, \delta)$ are induced vector fields by some arbitrary Lie algebra elements $A_{ij}\in T_e \G$ of some Lie group $\G$. Let $\vec{x} = (x^1,x^2, \dots, x^n)$ be a coordinate point of $\R^n$. Then, by the Definition \ref{def.induced.vector}, we have 

\begin{equation}
\label{eq:RotationRn}
    \mathcal{J}_{ij} = x^i \partial_j - x^j \partial_i \equiv \mathcal{J}^I_{ij} = \frac{d}{dt} e^{tA_{ij}}\at[\bigg]{t=0} \cdot \vec{x}, 
\end{equation}
where, in matrix notation,\footnote{The dot ``$\cdot$'' used throughout all the proof (and in the subsequent section) is intent to symbolizes matrix multiplication.} by a direct comparison of both sides in equation \eqref{eq:RotationRn}, we can conclude that, in terms of the matrix elements of the matrices $A_{ij}$, 
\begin{equation}
\label{eq:SO(N)gen}
     (A_{ij})_{kl} = \delta_{il} \delta_{jk} - \delta_{ik} \delta_{jl}.
\end{equation}

Next, suppose that the remaining Killing vector fields $\mathcal{P}_i = \partial_i$ of $(\R^n, \delta)$ are induced vector fields by other Lie algebra elements $T_i \in T_e \mathcal{G}$. Consider the vector space defined by 
\begin{equation}
\label{eq:spaceRn+1}
    \vec{x} \in \R^{n+1}; \qquad \vec{x} = (x^1, \dots, x^n, 1)
\end{equation}
and operations 
\begin{equation}
    \begin{split}
     \vec{x} + \vec{y} &:= (x^1 + y^1, \dots , x^n + y^n, 1),    \\
        \lambda \vec{x} &:= (\lambda x^1, \dots ,\lambda x^n,1).  
    \end{split}
\end{equation}
Clearly there is an isomorphism between $\R^n$ and $\R^{n+1}$ defined in this way. Using the same prescription adopted in equation \eqref{eq:RotationRn}, but now considering a coordinate point $\vec{x} \in \R^{n+1}$ as given in \eqref{eq:spaceRn+1}, we arrive at the following expression for the matrix elements of the $T_i$:
\begin{equation}
\label{eq:Transgen}
     (T_i)_{jk} = \delta_{ij} \delta_{n+1 k}, 
\end{equation}
with $j$ and $k$ in this special case running values from $1$ to $n+1$.

Here, one can recognize that the right-hand side of expression \eqref{eq:SO(N)gen} corresponds to the matrix elements of the generators of the Lie algebra $\mathfrak{so}(n)$ and, consequently, of the Lie group $SO(n)$. The subscript indices $i$ and $j$ on the left-hand side of \eqref{eq:SO(N)gen} represent an antisymmetric enumeration of the generators, while $k$ and $l$ label matrix rows and columns, respectively. On the other hand, the right-hand side of expression \eqref{eq:Transgen} corresponds to the matrix elements of the $i$-th Lie algebra generator of the translation group $T(n)$. Thus, by combining both the results obtained from the Killing vector fields $\mathcal{J}_{ij}$ and $\mathcal{P}_i$ of $(\R^n, \delta)$, and using the fact that the special Euclidean group $SE(n)$ can be represented as the semi-direct product $SE(n) \cong SO(n) \rtimes T(n)$, we can conclude that, when viewed as induced vector fields, they generate the entire group $SE(n)$.


We turn our attention now to the second part of the theorem. Let us suppose that the Killing vectors $\mathfrak{K}_i = x^i \partial_0 + x^0 \partial_i$ and $\mathfrak{J}_{ij} = x^i\partial_j - x^j \partial_i$ ($i$ and $j$ cyclic) of $(\R^{1+n}, \eta)$ are induced vector fields by some arbitrary Lie algebra elements $C_i$ and $B_{ij} \in T_e\mathcal{G}$, respectively. Then, following the same procedure applied in equation \eqref{eq:RotationRn}, but now considering a coordinate point $\vec{x} = (x^0, x^i) \in \R^{1+n}$, and using the Definition \ref{def.induced.vector}, we can find 
\begin{equation}
\label{eq:SO(1,3)gen}
  \begin{array}{cc}
     {(C_i)^\alpha}_\beta  = {\delta^\alpha}_0 \eta_{i\beta} - {\delta^\alpha}_ i \eta_{0\beta};     \\
  {(B_{ij})^\alpha}_\beta = {\delta^\alpha}_j \eta_{i \beta} - {\delta^\alpha}_i \eta_{j \beta}.     
  \end{array}
\end{equation}
Once again, from the right-hand side of the expressions in \eqref{eq:SO(1,3)gen}, we directly recognize that $C_i$ and $B_{ij}$ are the generators of the Lie algebra $\mathfrak{so}(1,n)$. The former is responsible for generating the boost sector of the group $SO(1,n)$, while the latter generates its rotation sector.

Finally, let us suppose that the Killing vectors $\mathfrak{P}_\mu = \partial_\mu$ of $(\R^{1+n}, \eta)$ are induced vector fields by some arbitrary Lie algebra elements $W_\mu \in T_e \mathcal{G}$ of an arbitrary group $\mathcal{G}$. Then, considering a vector space defined analogously to the expression \eqref{eq:spaceRn+1}, but now with $\vec{x} \in \R^{(1 + n) + 1 }$ and operations
\begin{equation}
\begin{split}    
     \vec{x} + \vec{y} &:= (x^0 + y^0, \dots, x^{n} + y^{n},1), \\
      \lambda \vec{x} &:= ( \lambda x^0, \dots, \lambda x^{n}, 1), 
\end{split}
\end{equation}
we can find that, following the same approach taken in equation \eqref{eq:RotationRn} and in terms of the matrix elements, 
\begin{equation}
    \label{eq:Transgenminko}
    (W_\mu)_{\alpha \beta} = \delta_{\mu \alpha} \delta_{n+1 \beta},
\end{equation}
where the indices $\alpha$ and $\beta$ in this case are taking values between $0$ and $n+1$. The right-hand side of equation \eqref{eq:Transgenminko} coincides with the matrix elements of the $\mu-$th generator of the spacetime translation group $ \mathcal{R}^{1,n}$. Therefore, using the fact the Poincaré group, which acts on $(\R^{1+n}, \eta)$, can be described as $ \mathcal{R}^{1,n} \rtimes SO(1, n)$, we can conclude that the Killing vector fields $\mathfrak{K}_i, \mathfrak{J}_{ij},$ and $\mathfrak{P}_\mu$ of $(\R^{1+n}, \eta)$, when viewed as induced vector fields according to Definition \ref{def.induced.vector}, generates the aforementioned Poincaré group.
\end{proof}


\section{Applications in the 3D Euclidean space and 4D Minkowski spacetime}\label{Sec3}

In physical grounds, the $3-$dimensional Euclidean space $(\R^3,\delta)$ and the Minkowski spacetime $(\R^{4}, \eta)$ have several applications across virtually all theoretical physics. The former servers as the foundation for physics ranging from Newtonian mechanics to non-relativistic electromagnetism. The latter is the pillar of special relativity and is also crucial for models of quantum field theory in flat spacetime. Thus, despite the construction presented in this work has a generalized character, this brief digression is more than necessary to explain our motivation for using these manifolds to exemplify our main result.

We start with $(\R^3, \delta)$. Let us suppose that $A_i \in T_e \G$ are elements of the Lie algebra corresponding to the Lie group $\G$ we intend to find. Supposing that the Killing vector fields $\mathcal{J}_i = \epsilon_{ijk}x^j \partial_k$ are induced vector fields. Then, by the Definition \ref{def.induced.vector}, are expressed as
\begin{equation}\label{estrela}
    \mathcal{J}_i \equiv \mathcal{J}^I_i = \frac{d}{dt} e^{tA_i}\at[\bigg]{t=0} \cdot \Vec{r},
\end{equation}
where $\Vec{r} = (x^1, x^2, x^3)$ are coordinates to an arbitrary element of the tangent plane $T_p \R^3 \cong \R^3$. For $i=1$ in \eqref{estrela}, we have
\begin{equation}
    \mathcal{J}_1= x^2 \partial_3 -x^3\partial_2 = 
    \begin{pmatrix}
        0 \\ -x^3 \\ x^2 
    \end{pmatrix} = 
    \frac{d}{dt} e^{tA_1}\at[\bigg]{t=0}
    \cdot
    \begin{pmatrix}
        x^1 \\ x^2 \\ x^3 
    \end{pmatrix}.
\end{equation}
Thus, we obtain
\begin{equation}
    A_1 = 
    \begin{pmatrix}
        0 & 0 & 0 \\ 
        0 & 0 & -1 \\
        0 & 1 & 0 
    \end{pmatrix}. 
\end{equation}
By repeating the same calculations for $i=2,\, 3$, we are led to the conclusion that $A_i = \tau_i$, where $\tau_i$ are the well known generators of the Lie algebra $\mathfrak{so}(3)$. In other words, the Killing vector fields can be regarded as the generators of the algebra $\mathfrak{so}(3)$, which, in turn, generates the group of rotations $SO(3)$. While the Killing vector fields act locally on tangent spaces on manifolds, they generate global rotations $R \in SO(3)$ that leave $\delta$ unchanged, that is, $\delta(R \Vec{r}_1, R \Vec{r}_2) = \delta(\Vec{r}_1, \Vec{r}_2)$.

The second instance follows a similar procedure to what we have employed thus far. Once again, we posit that the Killing vector fields $\mathfrak{K}_i = x^i \partial_0 + x^0\partial_i$ and $\mathfrak{J}_i = \epsilon_{ijk}x^j \partial_k$ of $(\R^{4}, \eta)$ are induced vector fields generated by Lie algebra elements of a Lie group $\G$ yet to be determined. In this case,
\begin{equation}
    \mathfrak{K}_i \equiv \mathfrak{K}^I_i = \frac{d}{dt}e^{tB_i}\bigg \vert_{t=0} \cdot \Vec{x}, \,\,\,\,
    \mathfrak{J}_i \equiv \mathfrak{J}^I_i = \frac{d}{dt}e^{tD_i}\bigg \vert_{t=0} \cdot \Vec{x}. 
\end{equation}
$B_i$ and $D_i$ are elements of the $T_e\G$ and $\Vec{x} = (x^0, x^1, x^2, x^3)$ represents an element in $T_{p} \R^{4} \cong \R^{4}$. For the particular case when $i=1$, we obtain $B_1$ and $D_1$. They are given by
\begin{equation}
    \mathfrak{K}_1 = x^0 \partial_1+ x^1 \partial_0 = \begin{pmatrix}
        x^1 \\ x^0 \\ 0 \\ 0
    \end{pmatrix} = 
    \frac{d}{dt}e^{tB_1}\bigg \vert_{t=0} \cdot 
    \begin{pmatrix}
        x^0 \\ x^1 \\ x^2 \\ x^3
    \end{pmatrix}
\end{equation}
which implies,
\begin{equation}
    B_1 =
    \begin{pmatrix}
        0 & 1 & 0 & 0 \\
        1 & 0 & 0 & 0 \\
        0 & 0 & 0 & 0 \\
        0 & 0 & 0 & 0 
    \end{pmatrix}. 
\end{equation}
Similarly, we have 
\begin{equation}
    \mathfrak{J}_1 = x^2 \partial_3 - x^3 \partial_2 = \begin{pmatrix}
        0 \\ 0 \\ -x^3 \\ x^2
    \end{pmatrix} = 
    \frac{d}{dt}e^{tD_1}\bigg \vert_{t=0} \cdot
    \begin{pmatrix}
        x^0 \\ x^1 \\ x^2 \\ x^3
    \end{pmatrix},
\end{equation}
from which, analogously, we find
\begin{equation} 
    D_1 = 
    \begin{pmatrix}
        0 & 0 & 0 & 0 \\
        0 & 0 & 0 & 0 \\
        0 & 0 & 0 & -1 \\
        0 & 0 & 1 & 0 
    \end{pmatrix}. 
\end{equation}

With the same argument and evaluation, running the values $i=2$ and $3$ we get
\begin{equation}
    B_i = K_i, \,\,\, D_i = J_i, 
\end{equation}
where $K_i$ are the boosts generators and $J_i$ are the rotation generators of the Lie algebra $\mathfrak{so}(1,3).$
Thus, we conclude that the Killing vector fields $\mathfrak{K}_i$ and $\mathfrak{J}_i$, from the perspective of induced vector fields, generates the group $SO(1,3)$. The latter, in its turn, is a group of global isometries (without space and time reversions) of $(\R^{4}, \eta)$ in the sense of preserving $\eta$: $\eta(\Lambda \Vec{x}, \Lambda \Vec{y}) = \eta( \Vec{x}, \Vec{y})$, for some $\Lambda \in SO(1,3)$.

\subsection{The translation sector}

Due to its linear structure, the Killing equations of $(\R^3,\delta)$ and $(\R^{4}, \eta)$ also admit linear combinations of $\partial_i$ and $\partial_\mu$ as solutions 
\begin{equation}
    T_{\R^3} = \varepsilon^i \partial_i, \qquad T_{\R^{4}} = a^\mu \partial_\mu,
\end{equation}
respectively. In terms of coordinates, the translations in both manifolds read,
\begin{equation}
    x^i \stackrel{T_{\R^3}}{\longmapsto} x'^i = x^i  + \varepsilon^i, \qquad x^\mu \stackrel{T_{\R^{4}}}{\longmapsto} x'^\mu  = x^\mu + a^\mu,
\end{equation}
which are not linearly realized. To address this issue, we add an extra dimension with no physical interpretation. To do so, consider the vector spaces defined by
\begin{equation}
\vec{x}\in \R^{4}; \quad \vec{x} = (x^1, x^2,x^3,1)    
\end{equation}
and operations, 
\begin{equation}
    \begin{split}
        \vec{x} + \vec{y} &:= (x^1 + y^1 , x^2  + y^2 , x^3 + y^3,1),  \\
        \lambda \vec{x}&  :=(\lambda x^1, \lambda x^2, \lambda x^3,1) 
    \end{split}
\end{equation}

Now, consider the linear operator represented by  
\begin{equation}
    \mathfrak{T}_{\R^3} = 
    \begin{pmatrix}
            1 & 0 & 0 & \varepsilon^1 \\
            0 & 1 & 0 & \varepsilon^2 \\
            0 & 0 & 1 & \varepsilon^3 \\
            0 & 0 & 0 & 1
    \end{pmatrix}.
\end{equation}
We have,
\begin{equation}
    \vec{x} \mapsto \vec{x}' = \mathfrak{T}_{\R^3}\vec{x} \iff x'^{i} = x^i  + \varepsilon^i.
\end{equation}
Analogously, for $(\R^{4},\eta)$ case, we introduce the five-dimensional space 
\begin{equation}
    \vec{x} \in  \R^5;\,\,\, \vec{x} = (x^0,x^1, x^2,x^3,1)
\end{equation}
with the same operations as before. Here, the linear operator that is intended to represent translations reads 
\begin{equation}
        \mathfrak{T}_{\R^{4}} =
    \begin{pmatrix}
    1 & 0 & 0 & 0& a^0 \\
    0 & 1& 0& 0& a^1 \\
    0& 0& 1 & 0& a^2 \\
    0 & 0& 0 & 1 & a^3\\
    0 & 0& 0&0  & 1  
    \end{pmatrix}.
\end{equation}
A direct calculation also shows that  
\begin{equation}
    \vec{x} \mapsto  \vec{x}' = \mathfrak{T}_{\R^{4}}\vec{x} \iff x'^{\mu} = x^\mu  +a^\mu.     
\end{equation}

Due to the nilpotent nature of the translation generators, we point out that $\mathfrak{T}_{\R^3} = e^{\varepsilon^i P_i}$, where $P_i$ are the matrices whose the only non-vanishing element is $1$, placed in the fourth column and $i$-th row (clearly a special case of the expression \eqref{eq:Transgen} for $n=3$). We also have $\mathfrak{T}_{\R^{4}} = e ^{a^\mu \widetilde{P}_\mu}$, with $\widetilde{P}_\mu$ are the matrices whose the only non-vanishing element is $1$, placed in the fifth column and $\mu$-th row (analogously of what we find in expression \eqref{eq:Transgenminko}, but for $n=3$). With the same prescription as before, identifying the parameterized Killing vector fields $T_{\R^3}$ and $T_{\R^{4}}$ as induced vector fields generated by a Lie algebra elements of a group to be obtained, we note that the differentials of $\mathfrak{T}_{\mathds{R}^3}$ and $\mathfrak{T}_{\mathds{R}^4}$, analogously to the Definition \ref{def.induced.vector}, do the job, respectively. In other words, by viewing the Killing vectors $T_{\R^3}$ and $T_{\R^{4}}$ as induced vector fields, we recover the generators $P_i$ of the translation group $T(3)$ and the generators $\widetilde{P_\mu}$ of the spacetime translation group $ \mathcal{R}^{1,3}$, respectively.  

In both cases presented in this example, by adding an extra (nonphysical) dimension to address the non-linearity of translations (as done more generally in the proof of the Theorem \ref{theo:1}), we can conclude the following: The Killing vector fields of $(\R^3, \delta)$, when defined as induced vector fields, generate the entire group of rotations and translations in the $3$-dimensional Euclidean space, i.e., the special Euclidean group $SE(3) \cong SO(3) \rtimes T(3)$. Similarly, the Killing vector fields of $(\R^{4}, \eta)$ generate the Poincaré group $\mathcal{R}^{1,3} \rtimes SO(1,3)$ which includes boosts, rotations, and translations in the Minkowski spacetime $(\R^{4}, \eta)$.

\section{Conclusion}\label{Sec6}

In this work, we have presented a previously unexplored method to recover Lie algebras, and consequently the corresponding Lie groups, from the Killing vector fields of the \(n\)-dimensional maximally symmetric spaces \((\R^n, \delta)\) and \((\R^{1+n}, \eta)\), relying solely on the very concept of induced vector fields. As stated in Theorem \ref{theo:1} and demonstrated throughout our construction, the definition of induced vector fields serves as a bridge, linking local isometries---derived from the Killing vector fields---to global symmetries, represented by the previously mentioned Lie groups $SE(n)$ and the Poincaré group, that act on these manifolds.

Finally, although the connection between Killing vectors and Lie groups is well established in the literature, as mentioned earlier, we hope that our work offers a new alternative approach to deriving this connection. Reinforcing once again, the bridge between local and global symmetries inherent in maximally symmetric spaces.

\begin{acknowledgments}
The work was partially supported by Coordenação de Aperfeiçoamento de Pessoal de Nível Superior - Brasil (CAPES) - Código de Financiamento 001.

\end{acknowledgments}

\end{document}